\documentclass{amsproc}

\usepackage[ansinew]{inputenc}
\usepackage{graphicx}
\usepackage{color}

\definecolor{rr}{rgb}{.8,0,.3}

\usepackage{enumerate,latexsym}
\usepackage{latexsym}
\usepackage{amsmath,amssymb}
\usepackage{graphicx}
\newfont{\bb}{msbm10 at 11pt}
\newfont{\bbsmall}{msbm8 at 8pt}

\def\R{\mathbb{R}}

\def\N{\mathbb{N}}

\newcommand{\ben}{\begin{enumerate}}
\newcommand{\bit}{\begin{itemize}}
\newcommand{\een}{\end{enumerate}}
\newcommand{\eit}{\end{itemize}}

\newcommand{\wt}{\widetilde}

\newcommand{\ed}{\end{document}}

\def\L{\Lambda}

\def\cA{\mathcal{A}}
\def\cU{\mathcal{U}}

\def\cS{\mathcal{S}}

\def\cR{\mathcal{R}}
\def\cB{\mathcal{B}}
\def\cW{\mathcal{W}}

\def\cL{\mathcal{L}}
\def\cM{\mathcal{M}}

\def\cG{\mathcal{G}}

\let\8=\infty \let\0=\emptyset  
\let\hat=\widehat

\let\landa=\lambda
\let\alfa=\alpha

\let\parc=\partial

\def\ep{\varepsilon}

\def\landa{\lambda}

\def\flecha{\rightarrow}
\def\esiz{\langle}
\def\esde{\rangle}

\def\cte.{\mathop{\rm cte.}\nolimits}

\def\N{\mathbb{N}}

\def\R{\mathbb{R}}

\def\H{\mathbb{H}}
\def\S{\mathbb{S}}

\def\X{\mathfrak{X}}

\newtheorem{theorem}{Theorem}[section]

\newtheorem{remark}[theorem]{Remark}
\newtheorem{corollary}[theorem]{Corollary}
\newtheorem{definition}[theorem]{Definition}

\newtheorem{example}[theorem]{Example}

\renewcommand{\sl}{\wt{\mathrm{SL}}(2,\R)}

\numberwithin{equation}{section}

\begin{document}

\begin{title}
[Uniqueness of immersed spheres]{Uniqueness of immersed spheres in three-manifolds}
\end{title}
\today
\author{José A. Gálvez}
\address{José A. Gálvez, Departamento de Geometría y Topología,
Universidad de Granada, 18071 Granada, Spain}
 \email{jagalvez@ugr.es}

\author{Pablo Mira}
\address{Pablo Mira, Departamento de Matemática Aplicada y Estadística, Universidad Politécnica de
Cartagena, 30203 Cartagena, Murcia, Spain.}

\email{pablo.mira@upct.es}

\thanks{The authors were partially supported by
MICINN-FEDER, Grant No. MTM2013-43970-P, Junta de Andalucía Grant No.
FQM325, 
and Programa de Apoyo
a la Investigacion, Fundacion Seneca-Agencia de Ciencia y
Tecnologia Region de Murcia, reference 19461/PI/14.}

\subjclass{Primary 53A10; Secondary 49Q05, 53C42}


\keywords{Immersed spheres, constant mean curvature, Hopf uniqueness problem, prescribed curvature, Weingarten surfaces, homogeneous three-manifolds.}


\begin{abstract}
Let $\cA$ be a class of immersed surfaces in a three-manifold $M$, and assume that $\cA$ is modeled by an elliptic PDE over each tangent plane. In this paper we solve the so-called Hopf uniqueness problem for the class $\cA$ under the only mild assumption of the existence of a  \emph{transitive} family of candidate surfaces $\cS\subset \cA$. Specifically, we prove that \emph{any compact immersed surface of genus zero in the class $\cA$ is a candidate sphere}. This theorem unifies and extends many previous uniqueness results of different contexts. As an application, we settle in the affirmative a 1956 conjecture by A.D. Alexandrov on the uniqueness of immersed spheres with prescribed curvatures in $\R^3$.
\end{abstract}

\maketitle

\section{Introduction}
Two deeply influential results model the geometry of compact constant mean curvature (CMC) surfaces in Euclidean three-space $\R^3$. 

First, Alexandrov's theorem \cite{A1} states that \emph{any compact embedded CMC surface in $\R^3$ is a round sphere}. The proof uses the maximum principle for elliptic PDEs, the existence of ambient reflections and a \emph{sweeping procedure}. The proof works in great generality for other classes of hypersurfaces in Riemannian manifolds, and also for elliptic PDEs, where it is known as the \emph{moving planes method} and has become one of the most important tools for proving uniqueness results (see \cite{GNN} and subsequent works). 

And second, Hopf's theorem \cite{Ho0} states that \emph{any immersed compact CMC surface of genus zero in $\R^3$ is a round sphere}. The proof uses holomorphicity of the so-called Hopf differential, or alternatively, the Poincaré-Hopf index theorem applied to principal line fields. The arguments are specific of dimension two. Hopf's theorem has been generalized to a large number of families of surfaces and geometric situations, but these generalizations usually need of many \emph{ad hoc} computations and considerations, depending on the chosen context. 

In this paper we prove an extremely general version of Hopf's theorem. In it, we substitute $\R^3$ by an arbitrary orientable three-manifold $M$, the class of CMC surfaces by an arbitrary class of immersed oriented surfaces $\cA$ in $M$ modeled by a (possibly fully nonlinear) elliptic PDE on each tangent plane, and the role of round spheres by the existence of a family of \emph{candidate examples} for which uniqueness is aimed. In these very general conditions, we prove that any surface of $\cA$ diffeomorphic to $\S^2$ is a candidate example (if no candidate example is diffeomorphic to $\S^2$, we obtain a non-existence result). This theorem somehow mimics the generality of the moving planes method by Alexandrov, but for the context of compact surfaces of genus zero with arbitrary self-intersections in three-manifolds.

When the ambient space $M$ is $\R^3$, or more generally a simply connected Riemannian homogeneous three-manifold, our general uniqueness theorem unifies and generalizes a large amount of the previously known theorems on uniqueness of compact surfaces of genus zero satisfying some elliptic geometric condition; see Section \ref{geoap} for the details.

There is an application of special relevance of our uniqueness study: an affirmative solution to a long-standing conjecture by A.D. Alexandrov \cite{A1} on the uniqueness of immersed spheres in $\R^3$ that satisfy a prescribed elliptic relation between their principal curvatures and Gauss map. See Section 5 for a specific statement of the conjecture (Theorem \ref{alexcon}) and its proof, as well as for an explanation of the historical context of this conjecture and its importance. As a byproduct, we solve another classical problem studied among others by Hopf, Chern, Alexandrov, Pogorelov or Hartman and Wintner: the classification of elliptic Weingarten spheres in $\R^3$. Specifically, we prove (Corollary \ref{weing}) that round spheres are the only immersed spheres in $\R^3$ whose principal curvatures $\kappa_1\geq\kappa_2$ satisfy a Weingarten relation $W(\kappa_1,\kappa_2)=0$, with $W\in C^1$ verifying the ellipticity condition $W_{\kappa_1}W_{\kappa_2}>0$ at umbilical points.

We state our main general result in Section 2, and prove it in Section 3. In Section 4 we explore some applications. Sections 5 is devoted to proving the Alexandrov conjecture in $\R^3$. 

\section{The general uniqueness theorem}

\subsection{Statement of the result}

Let $M$ be a smooth orientable three-manifold, and $\cA$ be a class of immersed oriented surfaces in $M$. By a \emph{sphere} or \emph{immersed sphere} in $M$ we will mean a compact, possibly self-intersecting, surface of genus zero immersed in $M$. We assume all surfaces to be smooth, i.e. of class $C^{\8}$.

Consider for every immersed oriented surface $\Sigma$ in $M$ its associated Legendrian lift $\cL_{\Sigma}: \Sigma \flecha G^2(M)$ into the Grassmannian of oriented $2$-planes in $TM$, which assigns to each $q\in \Sigma$ the value $\cL_{\Sigma} (q)= (q, T_q\Sigma)\in G^2(M)$.

\begin{definition}\label{defitransi}
Let $\cS$ be a family of immersed oriented surfaces in $M$. We say that $\cS$ is a \emph{transitive family} if the family of Legendrian lifts $\{\cL_S : S \in \cS\}$ satisfies:
 \begin{enumerate}
 \item
Each $\cL_S$ is an embedding into $G^2(M)$. 
\item
For every $(p,\Pi_p)\in G^2(M)$ there exists a unique $S=S(p,\Pi_p)\in \cS$ with $(p,\Pi_p)\in \cL_S$.
\item
 The family $\cS=\{S(p,\Pi_p):(p,\Pi_p)\in G^2(M)\}$ is $C^3$ with respect to $(p,\Pi_p)$.
 \end{enumerate}
\end{definition}

If we endow $M$ with a Riemannian metric $\esiz,\esde$, then $G^2(M)$ is naturally identified the unit tangent bundle $TU(M)$ of $(M,\esiz,\esde)$, and the map $\cL_{\Sigma}$ can be written as $\cL_{\Sigma}(q)= (q,N(q))$, where $N$ denotes the unit normal of $\Sigma$.

The family of all planes in $\R^3$, or the one of all horospheres in the hyperbolic three-space $\H^3$, are trivial examples of transitive families. Also, if $S\subset \R^3$ denotes a compact, strictly convex surface in $\R^3$, the family $\{p+S: p\in \R^3\}$ of all its translations also constitutes a transitive family in $\R^3$. Another example of transitive family $\cS$ in $\R^3$ is obtained as follows: consider a complete strictly convex graph $S\subset \R^3$ over a smooth convex bounded open region $D\subset \R^2$, and assume that $S$ converges $C^3$ asymptotically to the cylinder $\parc D  \times \R$. Then, the family $\cS$ given as the set whose elements are $S$, the cylinder $\parc D\times \R$, the $\pi$-rotation of $S$ with respect to some horizontal straight line, and all possible translations in $\R^3$ of these surfaces, is a transitive family in $\R^3$.

This type of constructions of transitive families using translations on the ambient space can be easily generalized to the context of \emph{metric Lie groups} (i.e. three-dimensional Lie groups endowed with some left invariant metric), by means of the concept of \emph{left invariant Gauss map} of a surface (see, e.g. \cite{GM,mpe11,MMP,mmpr0}). In particular, for specific applications, we will use in Section 4 the following construction:

\begin{example}\label{ejeintro}
\emph{Suppose that $M$ is diffeomorphic to $\R^3$ or $\S^3$, and endow $M$ with some metric Lie group structure. Let $S$ be an immersed sphere in $M$ with the property that its left invariant Gauss map $g:S\flecha \S^2$ with respect to this metric Lie group structure is a diffeomorphism. Then the family of $S$ and all of its left translations in $M$ is a transitive family in $M$. We observe that in this case the sphere $S$ may be non-embedded, see Remark 3.5 in \cite{MMP}}.
\end{example}

Informally, we will say that a class of immersed surfaces in $M$ is \emph{modeled by an elliptic PDE} if its elements are characterized locally as solutions to some elliptic PDE when viewed as graphs over each of their tangent planes. We formalize this idea next.

Let $\Phi=\Phi(x,y,z,p,q,r,s,t)\in C^{1,\alfa}(\cU)$, where $\cU\subset \R^8$ is an open convex set. Following \cite{Ho}, we will say that the second order PDE in two variables 
 \begin{equation}\label{fulpde}
 \Phi(x,y,u,u_x,u_y,u_{xx},u_{xy},u_{yy})=0
 \end{equation}
is \emph{absolutely elliptic} if $\Phi_r>0$ and $4\Phi_r \Phi_t - \Phi_s^2>0$ hold on $\cU$. Up to a change of variables, we will assume that the convex domain $\cU\subset \R^8$ intersects the vector space $x_1=\cdots = x_5=0$ of $\R^8$.

A function $u\in C^2(D)$ on some planar domain $D\subset \R^2$ is a solution to $\Phi=0$ if $$\{(x,y,u,u_x,u_y,u_{xx},u_{xy},u_{yy}): (x,y)\in D\}\subset \cU,$$ and $u=u(x,y)$ satisfies \eqref{fulpde}.

Let $\Upsilon$ denote the set of pairs $(p,\Pi)$, where $p\in M$ and $\Pi$ is an oriented $2$-plane in $T_pM$. Let us define for each $(p,\Pi)\in \Upsilon$ an absolutely elliptic PDE $\Phi_{(p,\Pi)}=0$, in terms of some local coordinates $(x,y,z)$ in $M$ around $p$, so that in these coordinates $p=(0,0,0)$ and $\Pi ={\rm span} \{\parc_x ({\bf 0}), \parc_y ({\bf 0})\}$ with its usual orientation. We call $(x,y,z)$ \emph{adapted coordinates} to $(p,\Pi)\in \Upsilon$. We will say that the family 
 \begin{equation}\label{defield}
\{\Phi_{(p,\Pi)}=0 : (p,\pi)\in \Upsilon\}
 \end{equation}
 is an \emph{elliptic PDE field on $M$}. Observe that we do not assume in general any continuity with respect to $(p,\Pi)$ in the choices of the absolutely elliptic PDEs $\Phi_{(p,\Pi)}=0$. 
 
A \emph{solution to the elliptic PDE field} \eqref{defield} is an immersed oriented surface $\Sigma$ in $M$ such that for every $q\in \Sigma$ we have that $\Sigma$ is a solution to the absolutely elliptic PDE $\Phi_{(q,T_q\Sigma)} =0$ when we view $\Sigma$ as a sufficiently small graph over $(q,T_q\Sigma)$; that is, when we view $\Sigma$ around $q$ as a graph $z=u(x,y)$ w.r.t. the coordinates $(x,y,z)$ adapted to $(q,T_q\Sigma)$, so that $u\in C^2(D(0;\epsilon))$ for $\epsilon>0$ small enough, and $u(0,0)=u_x(0,0)=u_y(0,0)=0$.

\begin{definition}\label{defieli}
We say that a class $\cA$ of immersed oriented surfaces in $M$ is \emph{modeled by an elliptic PDE} if its elements coincide with the space of solutions of an elliptic PDE field.
\end{definition}

One should note the generality of this definition: if a class of surfaces $\cA$ in a three-manifold $M$ is defined by some differential geometric condition that can be rewritten as an elliptic PDE over each tangent plane, then $\cA$ is modeled by an elliptic PDE in the sense of Definition \ref{defieli}. A very particular case is given by the surfaces of constant mean curvature $H\in \R$ in a Riemannian three-manifold $(M,g)$.

Our main theorem solves the Hopf uniqueness problem for arbitrary classes of surfaces modeled by elliptic PDEs in arbitrary orientable three-manifolds, provided there exists a transitive family of examples within the class:

\begin{theorem}\label{main1}
Let $\cA$ be a class of immersed oriented surfaces in $M$ such that:
 \begin{enumerate}
 \item
$\cA$ is modeled by an elliptic PDE.
 \item
There is a transitive family $\cS\subset \cA$. We will call them the \emph{candidate surfaces of $\cA$}.
\end{enumerate}
Then, any immersed sphere $\Sigma\in \cA$ is a candidate sphere, i.e. $\Sigma\in \cS$.
\end{theorem}
In particular, if the transitive family $\cS$ contains no spheres, this is a non-existence theorem. More generally, we will prove:

\begin{theorem}\label{main2}
Let $\cA$ be a class of immersed oriented surfaces in $M$ as in the statement of Theorem \ref{main1}.
Then there exists a symmetric $C^1$ bilinear form $\sigma$ globally defined on any $\Sigma\in \cA$, such that:
 \begin{enumerate}
\item
$\sigma$ vanishes at $p\in \Sigma$ if and only if $\Sigma$ has at $p$ a contact of order $k\geq 2$ with some $S\in \cS$.
 \item
$\sigma$ vanishes identically on $\Sigma$ if and only if $\Sigma$ is an open piece of some $S\in \cS$.
 \item
If $\sigma$ does not vanish identically on $\Sigma$, then $\sigma$ only has isolated zeros, and the null directions of $\sigma$ determine on $\Sigma$ two $C^1$ line fields with isolated singularities of negative index.
\end{enumerate}
\end{theorem}
Note that Theorem \ref{main1} follows from Theorem \ref{main2} by the Poincaré-Hopf theorem. Similarly, Theorem \ref{main2} implies that \emph{if $\cA$ is a class of surfaces as in Theorem \ref{main1}, then any compact surface of genus $g\geq 1$ of $\cA$ has at most $4g-4$ points where it has a contact of order $k\geq 2$ with some candidate surface $S\in \cS$}. This provides a wide extension of a well-known similar theorem on umbilics of compact CMC surfaces in $\R^3$.

\subsection{Discussion of the result}

Theorem \ref{main1} provides two sufficient conditions in order to solve the Hopf uniqueness problem for a class of surfaces $\cA$ in a three-manifold $M$. The ellipticity condition for $\cA$ is a standard geometric hypothesis, and without it (or some reformulation of it), uniqueness of spheres seems unlikely to hold. Thus, the only restrictive hypothesis in Theorem \ref{main1} is the \emph{existence} of a transitive family of surfaces $\cS\subset  \cA$. One should observe that without any hypothesis other than ellipticity for the class of surfaces $\cA$, the Hopf uniqueness problem for $\cA$ seems unclear to formulate, and hopeless to be solved in that generality. Also, one should observe that the existence of a transitive family $\cS\subset \cA$ implies, in particular, that the field of elliptic PDEs that defines $\cA$ has solutions.

For many important cases of surfaces modeled by elliptic PDEs, the existence of a transitive family is trivial. For example, for elliptic Weingarten functionals in $\R^3$, the family of round spheres of an adequate radius in $\R^3$ forms a transitive family. In many other natural uniqueness problems, the existence of a transitive family is basically given as a hypothesis (see for example the Alexandrov conjecture in Section \ref{alexsec}). In other prescribed curvature problems, transitive families can be obtained as the family of translations of a convex solution, and the existence of such convex solutions has been established by elliptic theory methods, see e.g. \cite{GG,P2}.

A case of special geometric interest consists of the classes of surfaces $\cA$ modeled by elliptic PDEs that are invariant under ambient rotations. This type of families includes, for instance, the case of elliptic Weingarten surfaces in rotationally symmetric homogeneous three-manifolds, and in particular the case of CMC surfaces in these spaces. A very influential result in this context is the theorem by U. Abresch and H. Rosenberg \cite{AR1,AR2} that \emph{immersed CMC spheres in rotationally symmetric homogeneous three-manifolds are spheres of revolution}.

Theorem \ref{main1} can be applied to this rotationally symmetric context, since in these conditions one can study the existence of a transitive family within the class by ODE analysis. This line of inquiry is studied by the authors in \cite{GM3}, which is a natural continuation of the present paper. Specifically, in \cite{GM3} we give a substantial generalization of the Abresch-Rosenberg theorem, going from CMC surfaces to more general classes of surfaces modeled by rotationally symmetric elliptic PDEs in homogeneous three-manifolds.

Nonetheless, in general, the problem of existence of a transitive family of solutions can be very hard for some classes of surfaces. In \cite{mmpr0} the second author proved jointly with Meeks, Pérez and Ros that if $X$ is a homogeneous manifold diffeomorphic to $\S^3$, then for every $H\in \R$ there exists an immersed sphere $S_H$ of constant mean curvature $H$, which is actually the only immersed sphere with constant mean curvature $H$ in $X$ up to congruence. The existence part in \cite{mmpr0} shows that, for each $H\in \R$, the family of spheres in $X$ congruent to $S_H$ constitutes a transitive family (since their left invariant Gauss maps in $X$ are diffeomorphisms). Once we know this existence, the uniqueness statement proved in \cite{mmpr0} can be obtained alternatively as a direct consequence of Theorem \ref{main1}.

\section{Proof of the general uniqueness theorem}\label{sec:mainth}
Let $\cA$ be a class of surfaces in the conditions of Theorem \ref{main1} in the orientable three-manifold $M$. We consider on $M$ an auxiliary, arbitrary, Riemannian metric $\esiz,\esde$. For each candidate surface $S\in \cS$ we consider its second fundamental form $II_S$ in $(M,\esiz,\esde)$. Then, the property that $\cS$ is a transitive family of surfaces allows to define on the unit tangent  bundle $TU(M)$ the map $\Lambda$ that assigns to each $(p,v) \in TU(M)$ the second fundamental form $II_S$ evaluated at $p$ of the unique candidate surface $S$ that has at $p$ the unit normal $v$. 

It is useful for our purposes to write this map $\Lambda$ in coordinates. Given $(q,w)\in TU(M)$, consider coordinates $(x,y,z)$ in $M$ around $q=(0,0,0)$ with $g_{ij}(q)=\delta_{ij}$ for the metric $\esiz,\esde$, and so that $w=(0,0,1)$ in the basis $\{\parc_x,\parc_y,\parc_z\}$ at the origin. In these coordinates we can consider some $\mathcal{O}\subset \R^3\times \S_+^2$ to be a neighborhood of $(q,w)$ in $TU(M)$. Note that if $(p,v)\in \mathcal{O}$, a basis for the tangent plane at $p$ of the candidate sphere $S$ determined by $(p,v)$ is 
 \begin{equation}\label{defe1}
E_1(p,v)= (1,0,-\eta_1 /\eta_3),\hspace{1cm} E_2 (p,v) = (0,1,-\eta_2 /\eta_3),\end{equation}
where $(\eta_1,\eta_2,\eta_3):= (v_1,v_2,v_3)\cdot \cG(p)$; here $\cG(p)=(g_{ij}(p))$ is the matrix of the metric at $p$. This lets us define the $2\times 2$ matrix $M_{II} (p,v)$ associated to the second fundamental form $II_S$ of $S$ at $p$ in terms of the basis $\{E_1(p,v),E_2(p,v)\}$. In this way, we can regard the map $\Lambda$ in coordinates as 
 \begin{equation}\label{landef}
 \Lambda : \mathcal{O}\subset \R^3\times \S_+^2 \flecha \cM_{2\times 2} (\R), \hspace{1cm} \Lambda (p,w)=M_{II}(p,w).
 \end{equation} By the regularity assumed for the family of candidate surfaces $\cS\subset \cA$, we see that $\Lambda$ is a $C^1$ map.

Let now $\psi:\Sigma\flecha M$ be any immersed oriented surface in $M$ with unit normal $N$, and let $\cL_{\psi}=(\psi,N):\Sigma\flecha TU(M)$ be its associated Legendrian immersion into $TU(M)$. Then, the map $\cM:=\Lambda \circ \cL_{\psi}$ defines a symmetric $C^1$ bilinear form $\cM$ on $\Sigma$.

Let $II$ denote the second fundamental form of $\Sigma$. Then we can define on $\Sigma$ the symmetric $C^1$ bilinear form 
\begin{equation}\label{sigmadef}
\sigma:= II - \cM : T \Sigma\times T \Sigma \flecha C^1(\Sigma).
\end{equation}
Two properties follow automatically from the definition:
 \begin{enumerate}
 \item
The tensor $\sigma$ vanishes identically on every candidate surface $S\in \cS$.
 \item
The tensor $\sigma$ vanishes at some point $q\in \Sigma$ if and only if $\Sigma$ has at $q$ a contact of order $k\geq 2$ with some candidate surface $S$.
 \end{enumerate}
Our objective is to understand the behavior of $\sigma$ around its zeros for surfaces $\Sigma\in \cA$. So, assume that $\Sigma\in \cA$, let $q\in \Sigma$ and consider the unique candidate surface $S$ that has an oriented first order contact point with $\Sigma$ at $q$. We can parametrize locally $\Sigma$ in a neighborhood of $q$ as an immersion $\psi(u,v)=(u,v,h(u,v))$ with respect to some local coordinates $(x,y,z)$ in $M$ around $q=(0,0,0)$, so that $\psi$ is defined on some disk $D(0;\delta)$, $\psi(0,0)=(0,0,0)$ and $\nabla h (0,0)= (0,0)$. Let $\overline{\nabla}$ denote the Riemannian connection of $M$ and $\cG:=(g_{ij})$ be the matrix of the metric $\esiz,\esde$ with respect to the coordinates $(x,y,z)$. Then, writing vectors in coordinates with respect to $\{\parc_x,\parc_y,\parc_z\}$, we obtain that the unit normal $N$ associated to $\psi$ is given by 
$$N=\frac{Z_1\times Z_2}{||Z_1\times Z_2||}, \hspace{1cm} Z_1 = \psi_u \cdot (\cG\circ \psi), \hspace{1cm} Z_2= \psi_v \cdot (\cG\circ \psi),$$ where $\times$ is the usual cross product in $\R^3$. Let $e=\esiz \overline{\nabla}_{\psi_u} \psi_u,N\esde$, $f= \esiz \overline{\nabla}_{\psi_u} \psi_v,N\esde$ and $g=\esiz \overline{\nabla}_{\psi_v} \psi_v,N\esde$ denote the coefficients of the second fundamental form $II$ of $\psi(u,v)$. A standard computation shows that there exist smooth functions $\Phi^e,\Phi^f,\Phi^g$, with $\Phi^e= \Phi^e(x,y,z,p,q,w)$, etc., such that 
 \begin{equation}\label{efg}
e= \Phi^e (u,v,h,h_u,h_v,h_{uu}), \hspace{0.5cm} f= \Phi^f (u,v,h,h_u,h_v,h_{uv}), \hspace{0.5cm} g= \Phi^g (u,v,h,h_u,h_v,h_{vv}).
\end{equation} 
Let now $\psi^*(u,v)=(u,v,h^*(u,v))$ be a parametrization of the candidate sphere $S$ around $q$, also with respect to the coordinates $(x,y,z)$ in $M$. Since $S,\Sigma$ have (at least) a first order contact point at $q$, we have $h^*(0,0)=0$, $\nabla h^*(0,0)=(0,0)$. 
The coefficients of the second fundamental form $e^*,f^*,g^*$ of $\psi^*$ then satisfy \eqref{efg} in terms of $h^*$ and its derivatives (but for the same functions $\Phi^e,\Phi^f,\Phi^g$). Thus,
 \begin{equation}\label{emenose}
 e-e^* = \Phi_1^e (h-h^*) + \Phi_2^e (h_u-h_u^*) + \Phi_3^e (h_v -h_v^*) + \Phi_4^e (h_{uu} - h_{uu}^*),
 \end{equation}
where $\Phi_1^e(u,v):=\int_0^1 \frac{\parc \Phi^e}{\parc z} (u,v,h^{\tau},h^{\tau}_u,h^{\tau}_v,h^{\tau}_{uu}) d\tau$ being $h^{\tau}(u,v):=(1-\tau) h(u,v) + \tau h^*(u,v)$, etc. The functions $\Phi_i^e (u,v)$ are smooth in a neighborhood of $(0,0)$. Similar formulas hold for $f-f^*$ and $g-g^*$. Since the unit normal of both $\psi$ and $\psi^*$ at $(0,0)$ is $(0,0,1)$, a computation shows that
 \begin{equation}\label{phi4}
\Phi_4^e (0,0)= \Phi_4^f (0,0)=\Phi_4^g (0,0)=1.
 \end{equation}

\begin{figure}[h]
\begin{center}
\includegraphics[width=10cm]{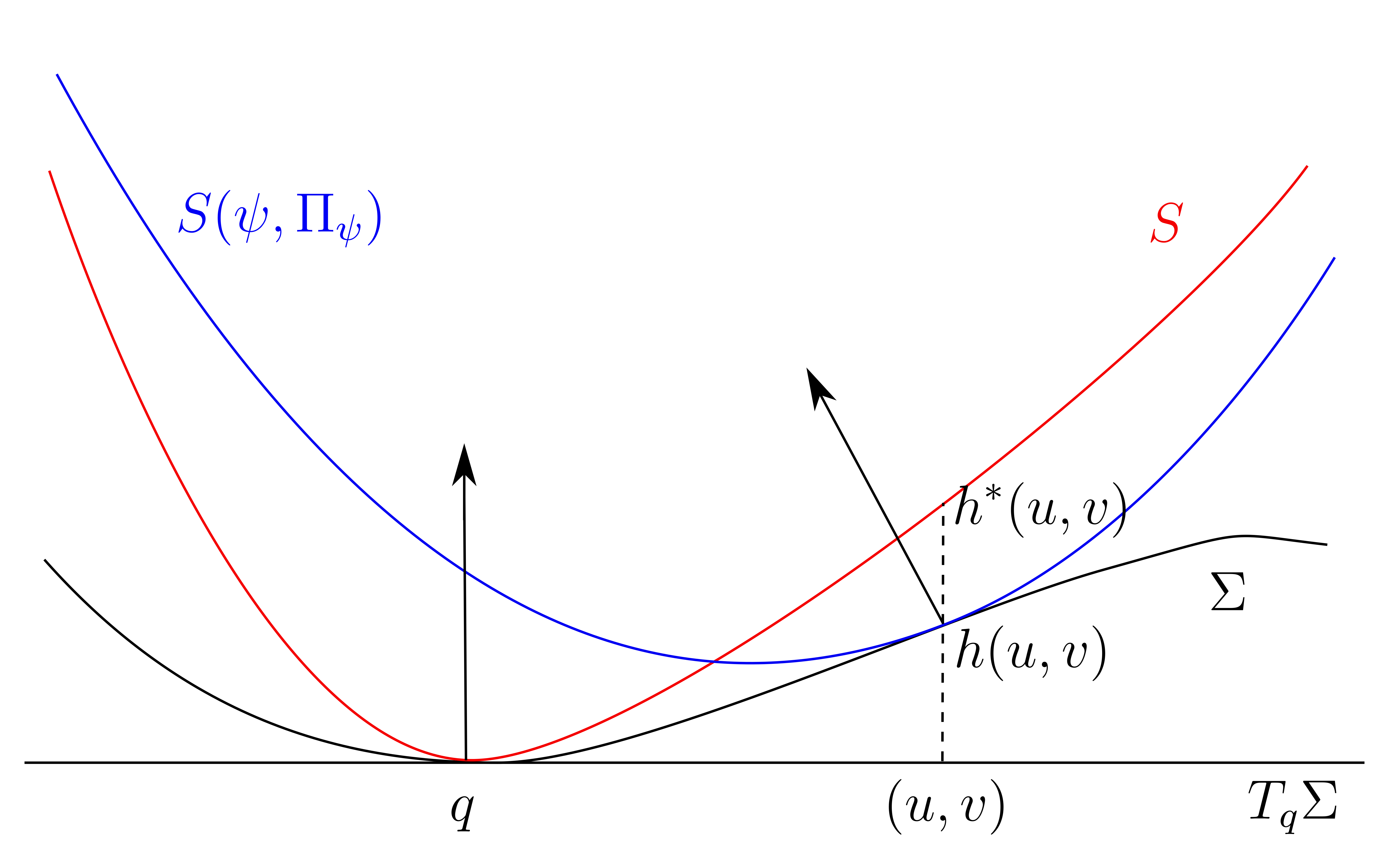}
\caption{Comparison of $\Sigma$ with its tangent candidate at each point.} \label{compara}
\end{center}
\end{figure}

Next, observe that as both $\Sigma,S$ belong to the class $\cA$, there exists a $C^{1,\alfa}$-smooth elliptic PDE \eqref{fulpde} with respect to the $(x,y,z)$ coordinates so that $\Sigma,S$ can be seen as graphs of solutions to \eqref{fulpde}. That is, the functions $h,h^*$ are solutions to \eqref{fulpde}. In these conditions, it is well known that the function $h-h^*$ satisfies a linear homogeneous elliptic PDE with $C^{0,\alfa}$ coefficients. Hence, after a linear change of coordinates in the $(u,v)$-variables, it follows from Bers' theorem \cite{Be} that if $h-h^*$ is not identically zero around the origin there exists a homogeneous harmonic polynomial $w(u,v)$ of degree $k\geq 2$ such that 
 \begin{equation}\label{hw}
h(u,v)-h^*(u,v)= w(u,v) + o(\sqrt{u^2+v^2})^k.
 \end{equation} 

Let now $A_{\sigma}$, $A_{II}$ and $A_{\cM}$ be the matrices associated to $\sigma$, $II$ and $\cM$ at the point $\psi(u,v)$ with respect to the $(\psi_u,\psi_v)$-basis, and let $B_{II}$ denote the matrix associated to $II_S$ at  the point $\psi^* (u,v)$ with respect to the $(\psi_u^*,\psi_v^*)$-basis, where $II_S$ is the second fundamental form of $S$. Then, from \eqref{sigmadef} we have 
 \begin{equation}\label{asifor}
A_{\sigma} = A_{II} - A_{\cM} = (A_{II} - B_{II}) + (B_{II} - A_{\cM}) :(u,v)\mapsto \cM_{2\times 2} (\R).
 \end{equation} 
 
Combining \eqref{emenose}, \eqref{phi4} and \eqref{hw} we deduce that
 \begin{equation}\label{1red}
 A_{II}-B_{II} = \left(\begin{array}{ccc} w_{uu} & w_{uv} \\ w_{uv} & w_{vv} \end{array} \right) + o(\sqrt{u^2+v^2})^{k-2}.
 \end{equation}

Let us now compute $B_{II}- A_{\cM}$. By definition, $\cM=\Lambda \circ \cL_{\psi}$. By the local representation in coordinates of $\Lambda$ in \eqref{landef}, we may regard $\Lambda\circ \cL_{\psi}$ as a map taking values on $\cM_{2\times 2} (\R)$, and whose value at each $p\in \Sigma$ is the matrix of the second fundamental form $II_S$ at $p$ of the candidate surface $S$ determined by $(p,N(p))$, with respect to the basis \eqref{defe1} of $T_pS$. Now, we can observe that the vector fields $\{E_1, E_2\}$ in \eqref{defe1}, when restricted to be vector fields along $\cL_{\psi}$, are precisely $\psi_u=(1,0,h_u)$ and $\psi_v =(0,1,h_v)$. All of this shows that there exists a $C^1$ map $\Psi (x,y,z,p,q)$ from a neighborhood of the origin in $\R^5$ into $\cM_{2\times 2}(\R)$, such that 
 \begin{equation}\label{foram}
 A_{\cM} = \Psi (u,v,h,h_u,h_v) .
 \end{equation}
Regarding $B_{II}$, we note that since $II_S=\cM$ on $S$, the same process as above shows that 
 \begin{equation}\label{forb2}
 B_{II} = \Psi (u,v,h^*,h_u^*,h_v^*),
 \end{equation}
 for the same $C^1$ function $\Psi$ defined previously. Thus, we obtain using \eqref{hw} that
  \begin{equation}\label{ter2}
 B_{II} - A_{\cM} = o (\sqrt{u^2+v^2})^{k-2}.
  \end{equation}
Putting together \eqref{asifor}, \eqref{1red}, \eqref{ter2} we conclude that either $h=h^*$ around $(0,0)$ (which implies $\sigma=0$ around $q$), or
 \begin{equation}\label{eqsig}
 A_{\sigma} = \left(\begin{array}{ccc} w_{uu} & w_{uv} \\ w_{uv} & w_{vv} \end{array} \right) + o(\sqrt{u^2+v^2})^{k-2} ,
 \end{equation} 
 where $w(u,v)$ is a harmonic homogeneous polynomial of degree $k\geq 2$. Let $\cU:=\{q\in \Sigma: \sigma(q)=0\}$. If $q\in \cU-{\rm int} (\cU)$, we see from \eqref{eqsig} that $q$ is isolated in $\cU$. From here, a connectedness argument shows that either $\sigma$ vanishes identically on $\Sigma$ and $\Sigma$ is an open subset of $S$, or else all points of $\cU$ are isolated. In this second case, \eqref{eqsig} yields that $\sigma$ is a Lorentzian metric on $\Sigma-\cU$, whose null directions at any $q\in \Sigma-\cU$ coincide with the asymptotic directions in $\R^3$ of the graph of the homogeneous harmonic polynomial $w(u,v)$, given by equation ${\rm Re} (w_{\zeta \zeta} d\zeta^2)=0$, where $\zeta:=u+iv$. By harmonicity of $w$, the index of these null directions around $q$ is negative. Thus, these null directions determine a pair of continuous line fields on $\Sigma-\cU$ that present isolated singularities of negative index at the points of $\cU$. This is impossible if $\Sigma$ is diffeomorphic to $\S^2$, by the Poincaré-Hopf index theorem. Thus, $\Sigma$ is an open piece of $S$. In particular, we have proved that $\sigma\equiv 0$ can only hold (even locally) if $\Sigma$ is an open subset of a candidate surface $S\in \cS$.
This proves Theorems \ref{main2} and \ref{main1}.

\section{Geometric applications}\label{geoap}

Let $S\subset \R^3$ be a closed, strictly convex surface, and suppose that $S$ is a solution to some geometric problem invariant by translations in $\R^3$, and whose solutions are determined by a $C^{1,\alfa}$ elliptic PDE over each tangent plane. Then, Theorem \ref{main1} implies that any immersed sphere $\Sigma$ which is also a solution to this geometric problem must be a translation of $S$. A similar claim remains true if we substitute $\R^3$ by $\H^3$ or $\S^3$. In this way, many uniqueness results of surface theory can be seen as direct consequences of Theorem \ref{main1}, among which we quote:

\begin{enumerate}
\item
The theorem by Hartman and Wintner \cite{HW1} that $W$-spheres in $\R^3$, $\S^3$ and $\H^3$ are round spheres ($W$-surfaces are defined by an elliptic $C^2$ Weingarten relation $W(H,K)=0$ between their mean and Gauss curvature). The Hartman-Wintner theorem is, in itself, a generalization of many classical uniqueness theorems, and in particular of Hopf's theorem for CMC spheres mentioned in the introduction.
 \item
The classical theorem by Alexandrov \cite{A1} on the uniqueness of ovaloids in $\R^3$ with a prescribed $C^2$ elliptic relation between its principal curvatures and its Gauss map, see \cite[Theorem 9]{A1} or \cite[Theorem 1]{GWZ}. This Alexandrov theorem, in particular, contains the uniqueness of the solution to the classical Christoffel and Minkowski problems for convex surfaces in $\R^3$; see Section \ref{alexsec} for more details.
 \item
The theorem by the authors \cite{GM} that if two immersed spheres $\Sigma_1,\Sigma_2$ in $\R^3$ have the same mean curvature at points with corresponding unit normals and $\Sigma_1$ is strictly convex, then $\Sigma_2$ is a translation of $\Sigma_1$.
 \item
The solution to the anisotropic Hopf problem by Koiso and Palmer \cite{KP}, i.e. the result that any sphere in $\R^3$ with constant anisotropic mean curvature is similar to the Wulff shape of the corresponding functional.
\end{enumerate}
Observe that Theorem \ref{main1} actually implies that the Koiso-Palmer theorem on uniqueness of the Wulff shape holds not only for the anisotropic mean curvature, but also for arbitrary smooth elliptic anisotropic Weingarten functionals.

The first three theorems mentioned above are unified and generalized in a substantial way by the general corollary to Theorem \ref{main1} that we present below, and that also contains as particular cases some previously known uniqueness results for CMC spheres in homogeneous three-manifolds (see e.g. Theorem 4.7 in \cite{DM} and the uniqueness statement of Theorem 4.1 in \cite{mmpr0}).

In the next corollary, $M$ is a simply connected homogeneous three-manifold not isometric to $\S^2(\kappa)\times \R$. Thus $M$ can be seen as a Lie group endowed with a left-invariant metric (see e.g. \cite{mpe11}) and, as explained in Section 2, we can define for any immersed surface $\Sigma$ in $M$ its \emph{left invariant Gauss map} $\eta:\Sigma\flecha \S^2$, where $\S^2$ is identified with the space of unit vectors of the Lie algebra of $M$. When $M$ is $\R^3$ with its usual abelian Lie group structure, $\eta$ is just the usual Gauss map for surfaces in $\R^3$. As usual, we denote by $\kappa_1\geq \kappa_2$ the principal curvatures of a surface in $M$.

\begin{corollary}\label{corapp}
Let $\Sigma_1,\Sigma_2$ be immersed spheres in $M$ satisfying the prescribed curvature equation 
 \begin{equation}\label{igualW}
 \cW(\kappa_1,\kappa_2,\eta)=0,
  \end{equation} 
where, for each fixed $\xi_0\in \S^2$, the function $\cW(x,y,\xi_0)$ is $C^{1,\alfa}$, symmetric (i.e. $\cW(x,y,\xi_0)=\cW(y,x,\xi_0)$) and satisfies the ellipticity condition
\begin{equation}\label{weinelipal}\frac{\parc \cW}{\parc x}\frac{\parc \cW}{\parc y} >0 .\end{equation}
Then, if the Gauss map of $\Sigma_1$ is a diffeomorphism, $\Sigma_2$ is a left translation of $\Sigma_1$.
\end{corollary}
\begin{proof}
Since the Gauss map of $\Sigma_1$ is a diffeomorphism, the family $\cS$ of all left translations of $\Sigma_1$ defines a transitive family $\cS$ in $M$ (see Example \ref{ejeintro}). Every $\Sigma\in \cS$ trivially satisfies \eqref{igualW}. In addition, because of the condition \eqref{weinelipal} and the symmetry of $\cW$, the equation \eqref{igualW} can be written locally over each tangent plane as a $C^{1,\alfa}$ elliptic PDE (this is explained, for instance, in Alexandrov \cite{A1}). Thus, the class of surfaces in $M$ that satisfy \eqref{igualW}  is modeled by an elliptic PDE as in Definition \ref{defitransi}, and contains the family $\cS$. By Theorem \ref{main1} any immersed sphere $\Sigma_2$ in this class is an element of the transitive family $\cS$, i.e. a left translation of $\Sigma_1$. 
\end{proof}

Further applications of Theorem \ref{main1} will be discussed elsewhere (see also \cite{GM3}).

\section{The Alexandrov conjecture}\label{alexsec}

A fundamental problem from classical differential geometry is the study of surfaces in $\R^3$ that satisfy some prescribed relation $\Phi(\kappa_1,\kappa_2,\eta)=0$ between their principal curvatures and Gauss map. If $\Phi$ is $C^1$ and the ellipticity condition $\Phi_{\kappa_1}\Phi_{\kappa_2}>0$ holds, the problem of classifying immersed spheres satisfying this equation has been deeply studied. In this sense, we may distinguish two lines of inquiry.

The first one goes back to the well-known Christoffel and Minkowski problems, and studies existence and uniqueness of ovaloids in $\R^3$ satisfying equation $\Phi(\kappa_1,\kappa_2,\eta)=0$. By combined efforts of Lewy \cite{Le}, Alexandrov \cite{A0,A1}, Pogorelov \cite{P1,P2} or Hartman and Wintner \cite{HW} among others, the uniqueness of immersed spheres that are solutions to $\Phi(\kappa_1,\kappa_2,\eta)=0$ is well understood if we restrict to closed spheres of positive curvature and elliptic functionals (see Remark \ref{remale} below). As regards some more recent important advances on this problem in arbitrary dimension, we may quote \cite{GLM,GMa,GMZ,GG} (see also \cite{EGM} for the case of hyperbolic space $\H^{n+1}$).

The second line of inquiry refers to the case $\Phi=\Phi(\kappa_1,\kappa_2)\in C^1$ and $\Phi_{\kappa_1}\Phi_{\kappa_2}>0$, which corresponds to the class of \emph{elliptic Weingarten surfaces} in $\R^3$, that obviously includes the case of CMC surfaces. The problem of whether round spheres in $\R^3$ are the only elliptic Weingarten spheres immersed in $\R^3$ has been studied by many authors, and partial solutions to this problem have been obtained, among others, by Hopf \cite{Ho}, Chern \cite{Ch}, Hartman and Wintner \cite{HW1}, Pogorelov \cite{P1,P2} or Alexandrov \cite{A1,A3}.

In the 1956 paper that opened his famous series \emph{Uniqueness theorems for surfaces in the large}, A.D. Alexandrov \cite{A1} conjectured a uniqueness result that would strongly generalize all these previously known theorems on uniqueness of spheres in $\R^3$. If true, the conjecture would imply that if a strictly convex sphere $S\subset \R^3$ satisfies an elliptic curvature equation $\Phi(\kappa_1,\kappa_2,\eta)=0$ (with $\Phi\in C^1$ not necessarily symmetric in $\kappa_1,\kappa_2)$, then any other \emph{immersed} sphere in $\R^3$ satisfying this equation is a translation of $S$. We prove this Alexandrov conjecture next.

\begin{theorem}[Alexandrov's conjecture]\label{alexcon}
Let $S\subset \R^3$ be a closed surface of positive curvature, and $\Phi=\Phi(k_1,k_2,x):\R^2\times \S^2\flecha \R$ be $C^{1}$ with respect to $k_1,k_2$, and such that \begin{equation}\label{elipal}\frac{\parc \Phi}{\parc k_1} \frac{\parc \Phi}{\parc k_2} >0\end{equation} on the subset $\Omega\subset \R^2\times \S^2$ given by 
\begin{equation}\label{omegal}\Omega =\{(\landa \kappa^0_1(p),\landa \kappa^0_2(p),\eta^0(p))\in \R^2\times \S^2 : p\in S, \landa \in \R\},
\end{equation}
where $\eta^0:S\flecha \S^2$ is the inner unit normal of $S$, and $\kappa^0_1\geq \kappa^0_2:S\flecha \R$ are its principal curvatures.

Let $f:\S^2\flecha \R$ be the function defined by
\begin{equation}\label{eqfilcero}
\Phi(\kappa^0_1(p),\kappa^0_2(p),\eta^0(p)) = f(\eta^0(p)) \hspace{1cm} \forall p\in S.
 \end{equation}
 
Then any other compact oriented surface of genus zero $\Sigma$ immersed in $\R^3$ whose Gauss map $\eta$ and principal curvatures $\kappa_1\geq \kappa_2$ satisfy 
\begin{equation}\label{eqfil}
\Phi(\kappa_1(p),\kappa_2(p),\eta(p)) = f(\eta(p)) \hspace{1cm} \forall p\in \Sigma
 \end{equation}
is a translation of $S$.
\end{theorem}
\begin{proof}
Fix $\Phi, f, S$ as in the statement. By convexity, the family $\cS=\{S+a : a\in \R^3\}$ constitutes a smooth transitive family of surfaces in $\R^3$. However, we cannot apply Theorem \ref{main1} as we did in Corollary \ref{corapp} because the class $\cB$ of surfaces $\Sigma$ in $\R^3$ that satisfy \eqref{eqfil} is not modeled by an elliptic PDE, for two reasons:
\begin{enumerate}
\item
The ellipticity condition \eqref{elipal} for \eqref{eqfil} only holds on the set $\Omega$ given by \eqref{omegal}.
\item
Equation \eqref{eqfil} is not of class $C^{1,\alfa}$. More specifically, even if $\Phi$ and $f$ were of class $C^{\8}$, equation \eqref{eqfil} would still be defined in terms of $\kappa_1,\kappa_2$, which are not necessarily $C^{1}$ at the umbilics of $\Sigma$. Hence, when we view $\Sigma$ as a graph $z=z(x,y)$ in local coordinates around an umbilic, the PDE in which \eqref{eqfil} translates would still not be $C^{1,\alfa}$ in general. This is a key problem.
\end{enumerate}
We should also point out that the function $\Phi(k_1,k_2,x)$ is not assumed to depend even continuously on $x$. For example, $\Phi(k_1,k_2,x)$ could be given by $k_1+k_2$ in some directions $x\in \S^2$ and by $k_1k_2$ in some others.

We circumvent these difficulties by modifying some arguments in the proof of Theorem \ref{main1}, and through the use of more sophisticated theorems for elliptic PDEs in dimension two. First of all, observe that since $S$ has positive curvature and is oriented by its inner unit normal, $II_S$ is positive definite. This implies that if $\Sigma$ is an immersed oriented surface in $\R^3$, the smooth symmetric bilinear form $\cM$ defined in the proof of Theorem \ref{main1} is a Riemannian metric on $\Sigma$.

Let now $\sigma:T\Sigma\times T\Sigma\flecha C^{\8}(\Sigma)$ be given by \eqref{sigmadef}, and let $T:\X(\Sigma)\flecha \X(\Sigma)$ denote the smooth endomorphism, self-adjoint with respect to $\cM$, given by 
 \begin{equation}\label{defT}
 \cM(T(X),Y)= II(X,Y) \hspace{1cm} \forall X,Y\in \X(\Sigma).
 \end{equation}
Note that a homothety $\frac{1}{\landa} S$ of $S$ also defines a transitive family of surfaces $\cS_{\landa}$ in $\R^3$ which changes $\cM$ to $\landa \cM$ on $\Sigma$. Then, it easily follows from \eqref{sigmadef}, \eqref{defT} that $T$ is proportional to the identity at a point $q\in \Sigma$ if and only if $\Sigma$ has at $q$ a contact of order at least two with some element of $\cS_{\landa}$ for some $\landa>0$. Let $\cU\subset \Sigma$ be the set of points where this happens. Then $T$ defines on $\Sigma\setminus \cU$ two smooth line fields, given by the eigendirections of $T$; they are mutually orthogonal with respect to $\cM$. Note that $\cU$ can be regarded as the set of singularities in $\Sigma$ of these smooth line fields.

We prove next that if $\Sigma$ satisfies \eqref{eqfil}, then $\cU=\{p\in \Sigma : \sigma(p)=0\}$. Indeed, let $p\in \cU$ and consider the translation of $S$ that has an oriented tangent contact with $\Sigma$ at $p$; we keep denoting this translated surface by $S$. Then $(\kappa_1(p),\kappa_2(p),\eta(p))= (\landa \kappa_1^0(p),\landa \kappa_2^0(p),\eta^0(p))$ for some $\landa>0$; in particular \eqref{elipal} holds at $p$. Noting that $\eta^0(p)=\eta(p)$, the monotonicity properties of $\Phi$ with respect to $k_1,k_2$ implied by \eqref{elipal} show that $$\Phi (\kappa_1(p),\kappa_2(p),\eta(p))=f(\eta(p))=\Phi(\kappa_1^0 (p),\kappa_2^0(p),\eta^0(p))$$ is impossible unless $\landa=1$, i.e. unless $S$ makes a contact of order at least two with $\Sigma$ at $p$. This is equivalent to $\sigma(p)=0$.

Take now $p\in \cU$, and assume that $\sigma\not\equiv 0$ around $p$. Since $\Sigma,S$ have a contact of order at least two at $p$ and $S$ has positive curvature, so does $\Sigma$. 
Consider coordinates $(x,y,z)$ in $\R^3$ so that $p=(0,0,0)$, and that $T_p\Sigma$ is the $z=0$ plane. Denote $\Sigma$ around the origin as an immersion $\psi=(\psi_1,\psi_2,\psi_3)$, with unit normal $\eta$. The classical \emph{Legendre transform} of $\psi$ is $$L_{\psi}= \left(-\frac{\eta_1}{\eta_3}, -\frac{\eta_2}{\eta_3}, -\psi_1 \frac{\eta_1}{\eta_3} - \psi_2 \frac{\eta_2}{\eta_3} - \psi_3\right).$$ $L_{\psi}$ is a graph around the origin, that we may write as $(x,y,h(x,y))$. If we now view the immersion $\psi$ as a map $\psi=\psi(x,y)$ in these coordinates, a computation shows that:
 \begin{enumerate}
 \item
$\psi(x,y)= (h_x,h_y,-h+x h_x + y h_y)$ and $\eta(x,y)=\frac{(-x,-y,1)}{\sqrt{1+x^2+y^2}}.$
 \item
The second fundamental form of $\psi$ is $II=\frac{1}{\sqrt{1+x^2+y^2}} \nabla^2 h$.
 \item
The principal curvatures of $\psi$ are given by $\kappa_i=\Psi^i(x,y,h_{xx},h_{xy},h_{yy})$, where
$\Psi^i (x,y,r,s,t)$, $i=1,2$, are continuous, and $C^1$ except on the set $\mathcal{R}\subset \R^5$
 \begin{equation}\label{lodedentro}(r(1+x^2) + t(1+y^2) + 2 s xy)^2 - 4 (rt-s^2)(1+x^2+y^2)=0.\end{equation} For every $(x,y)$ fixed, $\cR$ is a straight line in $\R^3$ (in $(r,s,t)$-coordinates) passing through the origin.
 \end{enumerate}
 
The same process we have just done for $\Sigma$ can be equally done for $S$; we denote the quantities associated to $S$ by $\psi^0$, $\eta^0$, $h^0(x,y)$, etc. Note, in particular, that $\eta(x,y)=\eta^0(x,y)=\frac{(-x,-y,1)}{\sqrt{1+x^2+y^2}}$. 

In this way, $\cM=\frac{1}{\sqrt{1+x^2+y^2}} \nabla^2 h^0$, and from \eqref{sigmadef}, 
\begin{equation}\label{sigmaro}
\sigma=\frac{1}{\sqrt{1+x^2+y^2}} \left(\begin{array}{ll} \varrho_{xx} & \varrho_{xy} \\ \varrho_{xy} & \varrho_{yy} \end{array}\right), \hspace{1cm} \varrho := h- h^0.
 \end{equation}
Note that the Hessian matrix of $\varrho$ is not identically zero, since $\sigma\not\equiv 0$ around $p$.

In addition, observe that we may rewrite equation \eqref{eqfil} as 
\begin{equation}\label{difef} 
F(x,y,h_{xx},h_{xy},h_{yy})= F(x,y,h_{xx}^0,h_{xy}^0,h_{yy}^0),
\end{equation}
where $F=F(x,y,r,s,t)$ is, for each $(x,y)$ fixed, of class $C^1$ with respect to $r,s,t$ except at points of $\mathcal{R}$.

Equation \eqref{eqfil} is uniformly elliptic on small compact regions where $0<c\leq \Phi_{k_i}\leq C$. For our parameters $(x,y)$, by \eqref{elipal}, this implies that there exists a closed ball $\cW \subset \R^5$ around $\alfa_0:=(0,0,h_{xx}(0,0),h_{xy}(0,0),h_{yy}(0,0))$ and positive constants $\hat{\Lambda}_i$ such that
\begin{equation}\label{unifeli2}
\hat{\Lambda}_1 (\xi_1^2+\xi_2^2) \leq \frac{\parc F}{\parc r}\xi_1^2 +  \frac{\parc F}{\parc s} \xi_1 \xi_2 + \frac{\parc F}{\parc t} \xi_2^2 \leq \hat{\Lambda}_2 (\xi_1^2+\xi_2^2),
\end{equation}
for every $(\xi_1,\xi_2)\in \R^2$ and at every point in $\mathcal{W}-\cR$.

Consider now a small disk $D\subset \R^2$ so that $\{(x,y,h_{xx},h_{xy},h_{yy}): (x,y)\in \overline{D}\} \subset \cW,$ and so that the same property holds for $(x,y,h_{xx}^0,h_{xy}^0,h_{yy}^0)$. In what follows we denote $r=h_{xx}(x,y)$, $r^0=h_{xx}^0(x,y)$, etc. We distinguish two possible cases for $(x,y)\in D$:

\vspace{0.2cm}

{\bf Case 1:} \emph{$(x,y,r,s,t)$ and $(x,y,r^0,s^0,t^0)$ do not both satisfy \eqref{lodedentro}}. In that case, the segment joining both points in $\cW$ meets the line $\cR$ defined by \eqref{lodedentro} in at most one point. Thus, by \eqref{difef}, $\varrho:=h-h^0$ satisfies at $(x,y)$ the linear PDE
 \begin{equation}\label{linro}
a_{11} \varrho_{xx} + 2 a_{12} \varrho_{xy} + a_{22}\varrho_{yy}=0,
 \end{equation}
where $a_{11}(x,y) = \int_0^1 \frac{\parc F}{\parc r} (x,y,r_{\tau},s_{\tau},t_{\tau}) d\tau,$ being $r_{\tau}:= (1-\tau) r + \tau r^0$, etc. 

\vspace{0.2cm}

{\bf Case 2:} \emph{$(x,y,r,s,t)$ and $(x,y,r^0,s^0,t^0)$ both satisfy \eqref{lodedentro}}, i.e. both lie in $\mathcal{R}$. Therefore, $(\kappa_1(x,y),\kappa_2(x,y) )=\landa (\kappa_1^0(x,y),\kappa_2^0(x,y))$ for some $\landa >0$. As $\Phi(\kappa_1,\kappa_2,\eta)=\Phi(\kappa_1^0,\kappa_2^0,\eta)$, the condition $\Phi_{k_1}\Phi_{k_2}>0$ implies $\landa=1$. This means that $\nabla^2 \varrho (x,y)= {\bf 0}$, so taking $a_{11}(x,y)=a_{22}(x,y)=1, a_{12}(x,y)=0$ we see that $\varrho:=h-h^0$ satisfies \eqref{linro} at  $(x,y)$. 

\vspace{0.2cm}

To sum up: $\varrho:=h-h^0$ satisfies in $D$ the linear PDE \eqref{linro}, where the coefficients $a_{ij}(x,y)$ might not be continuous but satisfy (by \eqref{unifeli2}) the uniform ellipticity condition 
 \begin{equation}\label{unifeli3}
\Lambda_1 (\xi_1^2+\xi_2^2) \leq \sum a_{ij}(x,y) \xi_i \xi_j \leq \Lambda_2 (\xi_1^2+\xi_2^2),
\end{equation}
for every $(x,y)\in D$ and every $(\xi_1,\xi_2)\in \R^2$, where $\Lambda_1:= {\rm min} \{1,\hat{\Lambda}_1\}$, $\Lambda_2:= {\rm max} \{1, \hat{\Lambda}_2\}$. 

From \eqref{unifeli3} we can clearly assume in \eqref{linro} that $a_{11} + a_{22} =2$, that $a_{11} >0$ and that $a_{11} a_{22}-a_{12}^2\geq c^2>0$ for some constant $c\in (0,1]$. 
 
 Let $\mu_1$ be given by $a_{11} = 1-\mu_1$ (thus, $a_{22} = 1+ \mu_1$), let $\mu_2 := -a_{12}$, and rewrite \eqref{linro} as 
 \begin{equation}\label{linro2}
 \varrho_{xx} + \varrho_{yy}= \mu_1 (\varrho_{xx} -\varrho_{yy}) + 2 \mu_2 \varrho_{xy}.
 \end{equation}
By $a_{11} a_{22}-a_{12}^2\geq c^2>0$ we have $\mu_1^2+\mu_2^2 \leq 1-c^2$, so by Cauchy-Schwarz,
 \begin{equation*}
(\varrho_{xx}+\varrho_{yy})^2 \leq (1-c^2)((\varrho_{xx}-\varrho_{yy})^2 + 4 \varrho_{xy}^2),
 \end{equation*}
which can be rewritten as 
 \begin{equation}\label{lineq2}
\varrho_{xx} \varrho_{yy}-\varrho_{xy}^2 \leq -\ep^2 (\varrho_{xx}+\varrho_{yy})^2, \hspace{1cm} \left(\ep^2= \frac{c^2}{4(1-c^2)}\right).
 \end{equation}

On the other hand, \eqref{linro} satisfies the conditions of the Bers-Nirenberg representation in \cite{BN}. As $S,\Sigma$ are smooth, $\varrho\in C^{\8}$, $\varrho \not\equiv 0$, so the Bers-Nirenberg strong unique continuation principle implies $\varrho(x,y)=w(x,y)+ o(\sqrt{x^2+y^2})^k$, where $w(x,y)$ is a homogeneous polynomial of degree $k\geq 3$, since $\Sigma,S$ have at $p$ a contact of order at least two. Therefore, from \eqref{lineq2}, 
 \begin{equation}\label{linp}
\frac{w_{xx} w_{yy}-w_{xy}^2}{(w_{xx}+w_{yy})^2} \leq -\ep^2 <0, \end{equation} unless $w_{xx}+w_{yy}=0$ (if $w_{xx}+w_{yy}=0$, we can argue directly that $w_{xx} w_{yy}-w_{xy}^2<0$ on $\R^2\setminus\{(0,0\}$).
By the fundamental theorem of algebra, 
 \begin{equation}\label{fta}
w(x,y)=w_0(x,y)\Pi_{j=1}^l (\alfa_j x + \beta_j y)^{m_j}
 \end{equation}
where $w_0(x,y)$ is a homogeneous polynomial without zeros, $m_j\in \N$, $m_j>0$, and any two vectors $(\alfa_i,\beta_i),(\alfa_j,\beta_j)$ are not collinear if $i\neq j$. A computation shows that if $m_j>1$ for some $j$, the left-hand side of \eqref{linp} is zero along the line $\alfa_jx+\beta_jy=0$. This contradicts \eqref{linp}. So, $m_j=1$ for all $j\in \{1,l\}$ in \eqref{fta}. In particular, $w_{xx},w_{xy},w_{yy}$ only vanish simultaneously at the origin, by homogeneity. By \eqref{linp}, this implies that $w_{xx} w_{yy} - w_{xy}^2<0$ on $\R^2\setminus\{(0,0)\}$. 

As a consequence, $\varrho_{xx} \varrho_{yy} - \varrho_{xy}^2<0$ on the punctured disk $D^*$, choosing $D$ smaller if necessary. By \eqref{sigmaro}, we deduce that $\sigma$ is a Lorentzian metric on $D^*$; in particular, the null fields of $\sigma$ are well defined on $D^*$ and have an isolated singularity at the origin. These null fields are given by 
\begin{equation}\label{nufi} \varrho_{xx} dx^2 + 2 \varrho_{xy}dx dy + \varrho_{yy} dy^2 =0,\end{equation} and have the same index around $(0,0)$. Note that since $\varrho_{xx} \varrho_{yy} - \varrho_{xy}^2<0$ on $D^*$, the origin is an isolated critical point of $\varrho_x$. It is then well known in index theory that the index of the null fields of \eqref{nufi} is non-positive if and only if the index of $\nabla \varrho_x:= (\varrho_{xx}, \varrho_{xy})$ is non-positive (see e.g. \cite[p.167]{Ho}).  As $\varrho$ is a solution to \eqref{linro}, it follows from the Bers-Nirenberg representation that $\varrho_x,\varrho_y$ do not have local extrema in $D$, see e.g. \cite[p.262]{BJS}. This implies by an easy geometric argument that the index of the gradient of $\varrho_x$ around the origin is non-positive. We conclude then that the index of each null field of $\sigma$ around $p$ is also non-positive. 

Note that, by a connectedness argument as in the proof of Theorem \ref{main1}, we have also shown that if $\sigma\not\equiv 0$ on $\Sigma$, the zeros of $\sigma$ are isolated. However, $\sigma$ is not a Lorentzian metric with isolated singularities on $\Sigma$, because the ellipticity condition \eqref{elipal} is not global on $\Sigma$. So, we do not work with the null fields of $\sigma$, but with the eigendirections of the operator $T$ instead. By our previous arguments, these eigendirections are defined on the complement of the zeros of $\sigma$, which are isolated. Note that by \eqref{defT} $$\cM((T-Id)(X),Y)=\sigma(X,Y).$$ A simple calculation using this equation shows that, at points where $\sigma$ is a Lorentzian metric, the eigendirections of $T-Id$ bisect (with respect to the Riemannian metric $\cM$)  the null directions of $\sigma$. This indicates that the index of the eigendirections of $T$ around any zero $p$ of $\sigma$ coincides with the index of the null fields of $\sigma$ at $p$, which is non-positive.

All of this together shows that in case $\sigma\not \equiv 0$ on $\Sigma$, the eigenspaces of $T$ provide two line fields on $\Sigma$ with isolated singularities of non-positive index. This is impossible by the Poincaré-Hopf theorem, since $\Sigma$ is diffeomorphic to $\S^2$. Therefore $\sigma$ vanishes identically, and this implies as in Theorem \ref{main1} that $\Sigma$ is a translation of $S$. This concludes the proof.
\end{proof}
\begin{remark}\label{remale}
Alexandrov conjectured the validity of Theorem \ref{alexcon} in page 350 of \cite{A1}, after announcing that Theorem \ref{alexcon} holds under the additional, ``\emph{certainly superfluous}" hypothesis that both $\Sigma,S$ are real analytic (see \cite[Theorem 8]{A1}). Alexandrov also announced (Theorem 9 in \cite{A1}) that Theorem \ref{alexcon} holds provided $\Sigma$ is also strictly convex, or at least is such that the spherical image of its region of negative curvature does not cover the sphere, and claimed that the proof was ``\emph{more complicated here than in the preceding cases}'' of theorems also announced there. However, he did not supply a proof for these two theorems in the series of papers that \cite{A1} initiated; this contrasts with the rest of results announced in \cite{A1}.

A proof of Theorem \ref{alexcon} in the particular case that $\Sigma$ is also strictly convex and the ellipticity condition \eqref{elipal} holds globally on $\R^2\times \S^2$ (but under weaker regularity conditions) was recently found by Guan, Wang and Zhang \cite{GWZ}, generalizing previous theorems by Alexandrov, Pogorelov and Hartman-Wintner (\cite{A0,A1,P1,HW}; see also \cite{GWZ} for a historical account of the problem). In \cite{GM} the authors proved Theorem \ref{alexcon} for the particular case of prescribed mean curvature, i.e. for the choice $\Phi(k_1,k_2,x)=k_1+k_2$.

Theorem \ref{alexcon} is not true if $S$ does not have positive curvature; see \cite[Example 5.1]{GM}.
\end{remark}

Recall that a \emph{Weingarten surface} in $\R^3$ is an immersed oriented surface whose principal curvatures $\kappa_1\geq \kappa_2$ satisfy a $C^1$ relation $W(\kappa_1,\kappa_2)=0$. If $W_{\kappa_1}W_{\kappa_2}>0$, we say that $\Sigma$ is an elliptic Weingarten surface, since this inequality is precisely the ellipticity condition for the PDE associated to $W(\kappa_1,\kappa_2)=0$.

An immediate consequence of Theorem \ref{alexcon} is the result that \emph{round spheres are the only elliptic Weingarten spheres immersed in $\R^3$}. This settles another classical problem. In the embedded case, the result was proved by Alexandrov via the reflection principle \cite{A1,A3}. For real analytic surfaces, the result follows from work by Voss \cite{Vo}.
In the non-analytic, non-embedded case, the result was previously known only for the special case that $W$ is symmetric, i.e. when $W(\kappa_1,\kappa_2)=W(\kappa_2,\kappa_1)$, after works of Hopf, Hartman-Wintner and Chern (see \cite{Ho,HW1,Ch}). This symmetry condition is equivalent to the possibility of rewriting $W(\kappa_1,\kappa_2)=0$ locally as $F(H,K)=0$ for some $C^1$ function $F$, where $H,K$ are the mean and Gaussian curvatures of the surface. Without this symmetry condition at umbilical points, the elliptic PDE associated to the equation $W(\kappa_1,\kappa_2)=0$ is not of class $C^1$ at these points, and the arguments of Hopf, Hartman-Wintner and Chern are not applicable. For example, the statement that \emph{spheres of radius one are the unique immersed spheres in $\R^3$ whose principal curvatures satisfy the relation $\kappa_1+2\kappa_2=3$} is not covered by all these classical theorems, but is a particular case of the Corollary below:
\begin{corollary}\label{weing}
Let $\Sigma$ be an immersed sphere in $\R^3$. Assume that, in a neighborhood of each umbilical point of $\Sigma$, the principal curvatures $\kappa_1\geq \kappa_2$ satisfy a $C^1$ relation $W(\kappa_1,\kappa_2)=0$ with $W_{\kappa_1}W_{\kappa_2}>0$. Then $\Sigma$ is a round sphere.

In particular, round spheres are the only elliptic Weingarten spheres immersed in $\R^3$.
\end{corollary}

\def\refname{References}

\end{document}